\documentclass[12pt]{amsart}
\usepackage{color,graphicx,array, amssymb, amscd,slashed,amsmath,amscd,latexsym,multirow,mathrsfs}
\usepackage{soul}
\usepackage{comment}
\usepackage[pdftex]{hyperref}
\usepackage{tikz}
\usepackage{tikz-cd}
\usetikzlibrary{fit,matrix}
\usepackage{enumitem}
\newlist{steps}{enumerate}{1}
\setlist[steps, 1]{label = Step \arabic*:}

\newtheorem{theorem}{Theorem}[section]
\newtheorem{corollary}{Corollary}[section]
\newtheorem{lemma}{Lemma}[section]
\newtheorem{proposition}{Proposition}[section]
\newtheorem{remark}{Remark}[section]

\newtheorem{definition}{Definition}[section]

\numberwithin{equation}{section}
\def\Vec#1{\mbox{\boldmath $#1$}}

\numberwithin{equation}{section}
\newcommand{\R}{\mathbb R}
\renewcommand{\H}{\mathbb H}
\newcommand{\C}{\mathbb C}
\newcommand{\D}{\mathbb D}
\newcommand{\id}{\operatorname{id}}
\newcommand{\SL}[1]{{\rm SL} (2, #1)}

\newcommand{\tr}{\operatorname{tr}}

\newcommand{\eq}[1]{\begin{align*}#1 \end{align*}}

\newcommand{\diag}{\operatorname{diag}}

\renewcommand{\d}{\mathrm{d}}
\newcommand{\dz}{\mathrm{d}z}
\newcommand{\dzb}{\mathrm{d}\bar z}
\newcommand{\Geo}{\mathrm{Geo}(\mathbb{H}^3)}

\newcommand{\ISU}{\mathrm {SU}(1, 1)}
\newcommand{\SU}{\mathrm {SU}(2)}

\newcommand{\LISU}{\Lambda \mathrm {SU}(1, 1)_{\tau}}

\newcommand{\Uone}{\mathrm U(1)}
\newcommand{\UH}{\mathrm{U}\mathbb{H}^3}
\newcommand{\Gr}{\mathrm{Gr}_{1,1}(\blue{\mathbb{E}^{1,3}})}
\newcommand{\LSL}{\Lambda\mathrm{SL}(2, \mathbb{C})_{\tau}}
\newcommand{\LSLU}{\Lambda\mathrm{SL}(2, \mathbb{C})}

\numberwithin{equation}{section}
\def\Vec#1{\mbox{\boldmath $#1$}}

\newcommand{\red}[1]{#1} 
\newcommand{\blue}[1]{#1} 
\makeatletter
\renewcommand\st[1]{\@bsphack\@esphack}
\makeatother 

\usepackage{amsmath}	
\begin{document}
\title{The complex landslide flow and the method of integrable systems}
\dedicatory{}
  \author[S.-P.~Kobayashi]{Shimpei Kobayashi}
 \address{Department of Mathematics, Hokkaido University, 
 Sapporo, 060-0810, Japan}
 \email{shimpei@math.sci.hokudai.ac.jp}
 \thanks{The author is partially supported by Kakenhi 22K03304.}
 \date{\today}
\pagestyle{plain}
\begin{abstract}
 We investigate a connection between the complex landslide flow, defined on a pair of 
 Teichm\"uller spaces, and the integrable system \st{method's}\red{approach}  
 to harmonic maps into a symmetric space.
 We will prove that the holonomy of the complex landslide flow 
 can be derived from the holonomy of the family of flat connections determined by
 a harmonic map into the hyperbolic two-space.
  \end{abstract}
\maketitle    
\section{Introduction and the main result}
\textbf{(A)} \st{The McMullen's complex earthquake flow is a natural  holomorphic extension of the Thurston's earthquake flow, while the \textit{complex landslide flow} introduced in represents a complex cyclic extension  of the complex earthquake flow.}
\red{McMullen's complex earthquake flow \cite{McM} extends Thurston's earthquake flow holomorphically, while the \textit{complex landslide flow} \cite{BMS:cyc1} provides its  cyclic complex extension.
 This flow arises from the composition of two geometric maps, as described below. Let $M$ denote a differentiable closed oriented surface}
 \st{This flow is defined by the composition of two geometric maps as follows$;$ Let $M$ be a differentiable closed oriented surface} of genus 
 greater than or equal to two.
 \red{Let $\mathcal{T}$ denote the Teichmüller space of hyperbolic metrics on $M$, and $\mathcal{P}$ denote the space of complex projective structures on $M$.}\st{, and let $\mathcal T$ the Teichm\"uller space parameterized by hyperbolic metrics on $M$.
 Moreover, let $\mathcal P$ be the space of complex projective structures on $M$.}
 \red{The complex landslide flow $P$ is defined by the following map:}
\st{The complex landslide flow $P$ is defined by }
\begin{equation}\label{eq:complexlandslide}
P:{\overline \D}^{\times} \times \mathcal T \times \mathcal T \to \mathcal P, \quad 
  (q, h, h^{\star}) \mapsto  SGr_s (L_{\theta} (h, h^{\star})),
 \end{equation}
  where $s = -\log |q|$ and $\theta= \arg q$, and moreover, 
$L:[0, 2 \pi) \times \mathcal T \times \mathcal T \to \mathcal T \times \mathcal T$ 
 and $SGr: [0,\infty) \times \mathcal T \times \mathcal T \to \mathcal P$ 
are  the \textit{landslide flow} and  the \textit{smooth grafting}
 \red{defined}\st{given in the} below. 

\textbf{The landslide:} Let $h, h^{\star} \in \mathcal T$ be hyperbolic metrics on $M$.
 Then there exists a unique bundle morphism $b: TM \to TM$ such that 
 \begin{enumerate}
 \item It is self-adjoint with respect to $h$.
 \item The determinant is $1$.
  \item It satisfies the Codazzi equation $\d^{\nabla} b = 0$, that is, 
  for  the Levi-Civita connection $\nabla$ of $h$ and vector fields $u, v$ 
  on $M$, $\d^{\nabla} b$ is defined by
 $(\d^{\nabla} b) (u, v) :=\nabla_u (bv)-\nabla_v (bu)- b[u, v]$.
  \item The metric given by $h(b\bullet, b\bullet)$ 
  is isometric to $h^\star$ by a diffeomorphism
   isotopic to the identity.
 \end{enumerate}
 The above bundle morphism $b$ has been called the \textit{Labourie operator} \cite{La}. 
 Then \st{for a given $\theta \in [0, 2 \pi)$} the landslide $L_{\theta}$, \red{where $\theta \in [0, 2 \pi)$,} on $\mathcal T \times \mathcal T$
 is defined by
 \begin{equation}\label{eq:landslide}
  L_{\theta} (h, h^{\star}) := (h(\beta_\theta \bullet, \beta_\theta\bullet), 
  h(\beta_{\theta+\pi} \bullet, \beta_{\theta+\pi}\bullet)),  \quad 
  \beta_{\theta} :=\cos (\theta/2) E + \sin (\theta/2) J b,
 \end{equation}
 where $\beta_{\theta}$ is a family of bundle morphism\red{s} \st{defined by}
 with the identity $E: TM \to TM$ and the complex structure $J$ with respect to $h$.\red{
 Clearly $\det \beta_\theta = 1$, and $\beta_\theta$ satisfies the Codazzi equation \cite[Lemma 3.2]{BMS:cyc1}. Additionally,}
 \st{It is evident that $\beta_\theta$ satisfies the Codazzi equation, and $\det \beta_\theta = 1$.} the metric \st{$h_{\theta}$ given by}
$h_{\theta}(u,v) = h(\beta_{\theta}u, \beta_{\theta}v)$ 
 is a hyperbolic metric on $M$ for all $\theta \in \mathbb R$ 
 and $L$ defines an $S^1$-flow on $\mathcal T \times \mathcal T$
 \cite[Theorem 1.8]{BMS:cyc1}, 
 that is, $L_{\theta^{\prime}} (L_{\theta} (h, h^{\star}))
  = L_{\theta^{\prime}+\theta} (h, h^{\star})$ holds. 
\red{Define $L^1$ as the composition of $L$ with the projection onto the first factor.}
 \st{Let $L^1$ be a map defined by the composition of $L$ with the projection on the first 
 factor.} Then $L^1_{e^{i\theta}}$  is analogous to \st{the} Thurston's earthquake map 
 $E:\mathcal T \times \mathcal {ML} \to \mathcal T$, where $\mathcal {ML}$
 is the space of measured lamination\red{s} on $M$.

\textbf{The smooth grafting:} For a given pair $(h, h^{\star}) \in \mathcal T \times 
\mathcal T$ and 
 a positive number $s>0$, the \textit{smooth grafting} $SGr_s$
 \red{is defined using a constant Gaussian curvature  $-1< K < 0$ (CGC) surface in hyperbolic $3$-space $\mathbb{H}^3$ and its Gauss map, \cite[Section 5]{BMS:cyc1}.}
\st{can be defined by a corresponding constant Gaussian curvature (CGC for short) surfaces with $-1< K < 0$ in the hyperbolic $3$-space $\mathbb H^3$
 and its Gauss map as follows.:} 
 \red{Let $b$ be the Labourie operator such that $h^\star = h(b \bullet, b \bullet)$. Define the metric $\mathrm{I} = \cosh^2 (s/2) h$ and the operator $B = -\tanh(s/2) b$.}
\st{Let $b$ be the Labourie operator as before, that is, 
 $h^\star= h(b \bullet, b \bullet)$ holds, and consider the metric $\mathrm{I} = \cosh^2 (s/2) h$
 and the operator $B= -\tanh (s/2) b$. }
 Then the propert\red{ies}\st{y} \red{(1) through (4)} of the Labourie operator impl\red{y}\st{ies} that 
 \st{the} there exists 
 a unique CGC  $-1< K=-1/\cosh^2(s/2)<0$ surface $f$ in $\mathbb H^3$ whose first fundamental form
  and shape operator are
 given by $\mathrm{I}$ and $B$, respectively. 
 \st{Moreover, consider a}\red{Take the} unique geodesic starting from a point $f(p) \in \mathbb H^3 (p \in M)$, with \st{the} 
 the unit normal \st{normal} $n(p)$ of $f(p)$ as initial velocity. 
 \red{Then by considering the point the geodesic intersecting with}\st{ending to} the boundary of $\mathbb H^3 \cong 
 \mathbb C P^1$, \st{gives the} a developing map $\operatorname{dev}_s: \widetilde M \to \mathbb C P^1$ 
 and \st{defines} a complex projective structure on $M$ \red{are given}, where 
 $\widetilde M$ denotes the universal cover of $M$. 
 \st{Thus} \red{In this way,} the smooth grafting $SGr_s: \mathcal T \times \mathcal T
  \to \mathcal P$ \st{has been}\red{is defined.}
  
  The complex landslide flow $P_q= SGr_s \circ L_{\theta}, \;
  (s = -\log|q|,  \; \theta = \arg q)$ shares many properties with the complex earthquake flow, 
  and in particular, it converges to the complex earthquake flow, 
  \cite{BMS:cyc1, BMS:cyc2}. \st{in details.}
  \red{In particular, in \cite[Theorem 5.1]{BMS:cyc1},} it \st{has been}\red{was} shown that the complex landslide $P$ \st{is}\red{was} actually a flow, and it \st{is}\red{was} holomorphic with respect to $q \in \D^{\times}$ by showing the holomorphicity of the holonomy of the projective structure. 
  See \st{also} Section \ref{sbsc:cland} below, \cite[Section 5]{BMS:cyc1} and 
  \cite[Proposition 5.14]{BMS:cyc1} \blue{or \cite{BE}} for more details. 

\textbf{(B)} \st{On the one hand,  a} \red{
 In \cite{IK:CGCH3},
 weakly complete CGC surfaces 
 with $-1< K < 0$ in $\mathbb H^3$ \red{were} classified 
 \st{in terms of}using} integrable system method\red{s}, \st{the so-called}\red{in particular} the \textit{loop group method}.
 The heart of the classification \st{is}\red{was} based on  harmonic maps from a Riemann surface $M$ into \st{the}\red{a} symmetric space, the hyperbolic $2$-space $\H^2 = \ISU/\Uone$ 
 and the spectral \red{parameter} deformation family in $\mathbb C^{\times}$.
 
 \textbf{Harmonic maps:} Recall that a harmonic map from a Riemann surface $M$
 into a symmetric space $G/K$ is characterized as follows, for example see \cite{DPW}.
 Let us consider a smooth map $g$ from a Riemann surface $M$ into a symmetric space $G/K$. Denote the universal cover of $M$ by  $\widetilde M$, take a lift $F: \widetilde M  \to G$, and let 
 $\alpha = F^{-1} \d F$ be the Maurer-Cartan form. Moreover, define a family \red{of connections} parameterized by $\lambda \in S^1 \subset \C$
 \begin{equation}\label{eq:familyconn}
 \nabla^{\lambda} = \d + \alpha^{\lambda}, \quad \mbox{with}\quad
\alpha^{\lambda} = \lambda^{-1}\alpha_{\mathfrak p}^{\prime} + \alpha_{\mathfrak k} + 
 \lambda \alpha_{\mathfrak p}^{\prime \prime}, 
 \end{equation}
 where $\alpha_{\mathfrak p}$ and $\alpha_{\mathfrak k}$
 denote the $\mathfrak p$- and  the $\mathfrak k$-parts
 of the direct \red{sum} decomposition of $\mathfrak g = \mathfrak k \oplus \mathfrak p$ for the Lie algebra $\operatorname{Lie} (G) = \mathfrak g$ and $\operatorname{Lie} (K) = \mathfrak k$, \red{the Cartan decomposition of 
 $\mathfrak g$}. 
 Moreover, $\prime$ and $\prime \prime$ denote the 
 $(1,0)$- and $(0,1)$-part with respect to the complex
 structure of the Riemann surface $M$. Then 
 $ \nabla^{\lambda}$ defines a family of connections 
 parameterized by $\lambda \in \mathbb C^{\times}$
 and  $\d+\alpha^{\lambda}|_{\lambda =1}$ gives a flat connection
 from the Maurer-Cartan equation $(\d \alpha^{\lambda} +
\alpha^{\lambda} \wedge \alpha^{\lambda})|_{\lambda =1} = 0$ for the frame $F$.
 It is well known that $g:M\to G/K$ is harmonic if and 
 only if $ \nabla^{\lambda}$ gives a family of \textit{flat connections} for any 
  $\lambda \in S^1$, that is, 
 $\d \alpha^{\lambda} +
\alpha^{\lambda} \wedge \alpha^{\lambda}= 0$
 holds for all  $\lambda \in S^1$.
 Then there exists a family of frames 
 \[
  F^{\lambda}:\widetilde{M} \to \Lambda G_{\tau}
\]
 such that $\alpha^{\lambda} = (F^{\lambda})^{-1}  \d F^{\lambda}$
 and it is called the \textit{extended frame} of the  harmonic map $g$.
 Here $\Lambda G_{\tau}$ denote\red{s} the loop group of $G$:
 \[
 \Lambda G_{\tau} = \{\gamma : S^1 \to G\mid \mbox{$\gamma$ is smooth and 
 $\gamma(-\lambda) = \tau \gamma(\lambda)$} \},
 \]
where $\tau$ is the involution \red{corresponding}\st{according} to the symmetric space $G/K$, that is, 
$(\operatorname{fix} \tau)_0 \subset K \subset \operatorname{fix} \tau$,
where $\operatorname{fix} \tau$ denotes the fixed point set of $\tau$ in $G$, 
 and the subscript zero denotes the identity component.
Moreover, introducing a suitable topology to  $\Lambda G_{\tau}$, 
it becomes \red{an} infinite-dimensional Banach Lie group, the so-called \textit{loop group} \cite{LoopGroup}.
It has been known \red{as} the Weierstrass type representation for such harmonic maps
\red{through the loop group decomposition of 
$\Lambda G_{\tau}$, see  \cite{DPW}.}
\red{Note that the harmonic map from a 
Riemann surface $M$ into $G/K$ depends on 
the complex structure on $M$ and the metric of $G/K$, which can be induced by the Killing form of $G$. 
 However the complex structure on $M$ for the harmonic map is not directly related to a
pair of points $\mathcal T \times \mathcal T$
to define the complex landslide in (A). }

  \textbf{The spectral \red{parameter} deformation:} In our case, $G/K$
  is the hyperbolic two space $\mathbb H^2$ \red{which can be represented by  the degree $2$
  indefinite special unitary group $G =\ISU$ and the isotropy subgroup $K = \Uone=\{ \operatorname{diag} (e^{i t}, e^{-i t})\mid 
   t  \in \R\} \subset \ISU$.} \st{Moreover,}\red{Note that} the involution $\tau$ is explicitly 
 given by $\tau F = \operatorname{Ad} \operatorname{diag}(1, -1) F$ for 
  $F \in \ISU$. Let us consider a harmonic map from $M$ into $\mathbb H^2$
 and take the extended frame $F^{\lambda}:\widetilde{M} \to \LISU$ as above.
 If we evaluate the extended frame
 $F^{\lambda}$ at $\lambda \in \mathbb C^{\times}\setminus S^1$ then
 it takes values in $\LSL$ not in $\LISU$. In \cite[Theorem 2.1]{IK:CGCH3},
it \st{has been}\red{was} shown that the extended frame $F^{\lambda}$ \red{gave}\st{gives} a family of CGC surfaces \st{with $-1< K <0$} in $\mathbb H^3$ as
  \[
  \blue{\hat f_q}= F^{\lambda} (F^{\lambda})^*|_{\lambda= q,} \quad (q \in \mathbb D^
  \times),
  \]
 where the upper subscript $*$ denotes a composition of the complex conjugation and the transpose,
 and the constant Gaussian curvature \red{$K$ of $\hat f_q$} is given by 
 \[
K = - \left(\frac{2 |q|}{|q|^2 + 1}\right)^2 \in (-1, 0),
\quad (|q|  \in (0, 1)).
\]
 Note here that the hyperbolic three-space $\mathbb H^3$ is realized by the quadric in the complex Hermitian $2$\red{-}by\red{-}$2$ matrices, see Section \ref{sbsc:Gauss}.
 The evaluation at $q\in \mathbb D^{\times}$
 has been called the \textit{spectral 
 \red{parameter} deformation} of the extended frame \red{\cite[Section 2.4]{IK:CGCH3}}.
 This construction was the heart of the classification of weakly complete CGC $-1<K <0$ 
 surfaces in $\mathbb H^3$\st{, see CGCH3}.
 Moreover, for a fixed $|q| \in (0,1)$, $\blue{\hat f_{q}}$ gives the
 $S^1$-family of CGC surfaces with a fixed
 \red{Gaussian curvature in} $-1< K = - 4 |q|^2/(|q|^2 + 1)^2 <0$ in $\mathbb H^3$ and it  
 becomes the \textit{associated family} of a CGC $-1< K <0$ surface in $\mathbb H^3$.
 \red{The scaled first and third fundamental forms of $\hat f_q$ are 
 hyperbolic metrics.}\st{, see Lemma for details.}
 \red{In this way,  we obtain a point 
 in  $\mathcal T \times \mathcal T$
 from the harmonic map from $M$ into $\mathbb H^2$.}

\textbf{(C)} From the discussions of (A) and (B),
 it is natural to expect a certain relation between the 
 above constructions. In both constructions, the $S^1$-families
 and the scaling families naturally appear, and in particular 
 the landslide and the smooth grafting look 
 similar to the associated family and the spectral \red{parameter} deformation, respectively.
 \st{In fact,} The following theorem is the main result of the paper.
 \begin{theorem}\label{thm:main}
 \red{Let $P_q(h, h^{\star})\; (q \in {\overline \D}^{\times}, (h, h^{\star}) \in 
 \mathcal T\times \mathcal T)$ 
 be the complex landslide flow. Then there 
 exists a family of flat connections $\nabla^{\lambda}$ 
 of a harmonic map into $\mathbb H^2$ corresponding to $(h, h^{\star})$
  such that 
 the holonomy of $P_q(h, h^{\star})$ is given by 
 the holonomy of $\nabla^{\lambda}$
 evaluated at $\lambda = \sqrt{q} \in {\overline \D}^{\times}$. 
 }
\end{theorem}

 We will give a proof in Section \ref{sc:proof}. 
\red{Note that} the extend\red{ed} frame $F^{\lambda}$ has a holomorphic dependence on the spectral parameter $\lambda \in \mathbb C^{\times}$.
  In particular, it is defined on  ${\overline \D}^{\times}$. Then, by noting
  Remark \ref{rm:main} (1), the holomorphic dependence of the holonomy of the complex landslide flow follows, and the following 
 statement holds.
\begin{corollary}\label{coro:landslide}
 \red{
 The map $q \mapsto P_q(h, h^{\star})$, from $\D^{\times}$ to $\mathcal{P}$, is holomorphic.
  }
 \end{corollary}
\begin{remark}\label{rm:main}
\mbox{}
\begin{enumerate}

\item[{\rm(1)}] Because we utilized the twisted loop group formulation 
 of a harmonic map into 
a symmetric space, the square root $\lambda = \sqrt{q}$ in {\rm Theorem \ref{thm:main}} arises.
 However, using the untwisted loop group formulation as demonstrated in \cite{BRS:Min}, it can be 
 expressed by the evaluation as $\lambda = q$,
 see the proof in {\rm Section \ref{sc:proof}}.
 Then {\rm Corollary \ref{coro:landslide}} follows immediately.
 \item[{\rm(2)}] {\rm Corollary \ref{coro:landslide}} has been proved in \cite[Theorem 5.1]{BMS:cyc1}
 by a direct computation of derivatives of the holonomy and by showing that the Cauchy-Riemann equation holds,
 however, {\rm Theorem \ref{thm:main}} gives a natural correspondence between 
 the complex landslide flow and the family of flat connections.
\item[{\rm(3)}]  A relation between the harmonic map and the complex landslide has been also 
 discussed in  \cite{BMS:cyc1}, and the spectral \red{parameter} deformation family in this paper will 
 give a natural understanding of such relation through the integrable systems approach.
 The key observation is a relation  in {\rm Lemma \ref{lem:J}}
 of two complex structures on $M$, namely,
 the $J$ and $i=\sqrt{-1}$ induced from the first and second fundamental forms
 of a CGC surface, respectively.

\end{enumerate}
\end{remark}
 The paper is organized as follows\st{:}\red{.} Preliminaries will be given in Section \ref{sc:Pre}
 and the proof of Theorem \ref{thm:main} will be given in Section \ref{sc:proof}.
 
{\bf Acknowledgments:}
\red{I would like to express my sincere gratitude to the anonymous reviewers for their valuable comments and suggestions, which have greatly contributed to improving the quality of this paper.}

\section{Preliminaries}\label{sc:Pre}
In this section, according to \cite{IK:CGCH3} we will recall surfaces in the hyperbolic three-space $\mathbb H^3$ and 
two Gauss maps, that is, the Legendrian and the Lagrangian Gauss maps, respectively.
Then by the spectral \red{parameter} deformation of the Lagrangian Gauss map, we will relate the extended 
frame of a harmonic map into hyperbolic two-space $\mathbb H^2$ and 
a constant Gaussian curvature $-1< K<0$ surface in $\mathbb H^3$.
Moreover, we will review the complex landslide and a CGC $-1< K<0$ surface in $\mathbb H^3$
according to \cite{BMS:cyc1}.

\subsection{Surfaces in the hyperbolic three-space}\label{sbsc:Gauss}
Let
\[
\boldsymbol{e}_0=
\begin{pmatrix}
1 & 0\\
0 & 1
\end{pmatrix}
\ \
\boldsymbol{e}_1=
\begin{pmatrix}
1 & 0\\
0 & -1
\end{pmatrix},
\
\boldsymbol{e}_2=
\begin{pmatrix}
0 & -i\\
i & 0
\end{pmatrix}
\ \mbox{and} \,\
\boldsymbol{e}_3=
\begin{pmatrix}
0 & 1\\
1 & 0
\end{pmatrix}.
\]
 The $4$-dimensional Minkowski space $\mathbb E^{1,3}$ 
 can be identified with the complex Hermitian $2$\red{-}by\red{-}$2$ matrices 
\[
\operatorname{Her}(2, \mathbb C) = 
\left \{
  \blue{\Vec{\xi}}= \sum_{j=0}^3\xi_j \Vec{e}_j \mid \xi_j \in \mathbb R
\right\} \longleftrightarrow \mathbb E^{1,3} = \left\{(\xi_0, \xi_1, \xi_2, \xi_3) \mid 
 \xi_0, \xi_1, \xi_2, \xi_3 \in \R \right\}.
\]
Then the standard inner product of $\mathbb E^{1,3}$ 
is given by $\langle \blue{\Vec{\xi}}, \blue{\Vec{\eta}}\rangle = - \frac12 \operatorname{tr} (\blue{\Vec{\xi}} \Vec{e}_2 {}^t \blue{\Vec{\xi}} \Vec{e}_2)$
 and $\langle \blue{\Vec{\xi}}, \blue{\Vec{\xi}}\rangle= - \det \blue{\Vec{\xi}} $
under the identification of $\mathbb E^{1,3}\cong \operatorname{Her}(2, \mathbb C)$. 
 The \textit{hyperbolic $3$-space} $\mathbb H^3$ 
 is defined by the unit central hyperquadrics in $\mathbb E^{1,3}
  \cong \operatorname{Her}(2, \mathbb C) $:
\[
\mathbb H^3 = \left\{
\blue{\Vec{\xi}} \in \operatorname{Her}(2, \mathbb C)  \mid 
\det \blue{\Vec{\xi}} = 1, \, \tr \blue{\Vec{\xi}} >0
\right\}.
\]
The special linear group of degree two $\SL{\mathbb C}$ acts isometrically and transitively on $\mathbb H^3$
via the action $(g, \blue{\Vec{\xi}})$ as $g \blue{\Vec{\xi}} g^*$, where the subscript $*$ denotes 
$g^*= {}^t \bar g$ for $g \in \SL{\mathbb C}$. The isotropy subgroup of this action at $\Vec{e}_0 = (1, 0, 0, 0) \cong \id$ is the special unitary group $\SU$ and hence $\mathbb H^3$ can 
 be represented as a homogeneous Riemannian symmetric space $\mathbb H^3 = \SL{\mathbb C}/\SU$. The natural projection $\pi: \SL{\mathbb C} \to 
 \mathbb H^3$ is given by $\pi (g) = g g^*$ and thus the hyperbolic $3$-space can be represented as $\mathbb H^3 =\{ g g^* \mid g \in \SL{\mathbb C}\}$.

 Let $f: M \to \mathbb H^3$ be a oriented surface and assume that its Gaussian curvature 
 satisfies $K>-1$.
 Then by the formula $K = -1 + \det B$, where $B$ is the shape operator of 
 $f$, $\det B>0$ follows. Thus the second fundamental form $\textrm{I\!I}$ defines a 
 Riemannian metric \st{and thus the unique}\red{which induces a} complex structure \st{follows}\red{on $f$}. Let us denote the 
 complex coordinate induced from the second fundamental form by $z = x + y i$.
 Then the fundamental forms can be represented as follows; Set $\sigma = \det B$ and
 the unit normal of $f$ as $n$. The first, second and third fundamental forms can be 
 computed by 
 \begin{align}
 \mathrm{I}&=\langle \d f, \d f\rangle=  Q \> \dz^2+(e^u +|Q|^2e^{-u})\> \dz \dzb
+\bar Q\>\dzb^2, \label{eq:1}\\
 \mathrm{I\!I}&=\langle \d f, \d  n\rangle=\sigma (e^u -|Q|^2e^{-u})\>\dz \dzb,  \label{eq:2}\\
 \mathrm{I\!I\!I} &=\langle \d n, \d n\rangle=\sigma^2\left(-Q \> \dz^2+(e^u +|Q|^2e^{-u})\> \dz \dzb
-\bar Q\>\dzb^2\right),\label{eq:3}
 \end{align}
  where $u : M \to \R$ is a smooth function and $Q \,\dz^2$ is the 
  $(2, 0)$-part of the first fundamental form with respect to the conformal structure 
 of the second fundamental form, which will be called the \textit{Klotz differential}.
 Note that the mean curvature $H$  and the Gaussian curvature $K$ of the surface $f$ can be represented by 
\begin{equation}\label{eq:H}
H = \frac{\sigma}2\, \frac{e^{2u} +|Q|^2}{e^{2u} -|Q|^2} \quad \mbox{and} \quad 
K = -1 + \sigma^2,
\end{equation}
 see \cite[\red{Section 1.1}]{IK:CGCH3} for more details.
\subsection{Natural projections, two Gauss maps and \red{a} Ruh-Vilms type theorem}
 We have several natural projections from the unit tangent sphere bundle 
 of $\mathbb H^3$
\eq{
\mathrm{U}\mathbb{H}^{3}=
\{ (\Vec{x},\Vec{v})\in \mathrm{Her}_2\mathbb{C}\times 
\mathrm{Her}_2\mathbb{C}
\
\vert
\
\det \Vec{x}=1, \mathrm{tr}\>\Vec{x}>0,\>
\det \Vec{v}=-1, \> \langle \Vec{x},\Vec{v}\rangle =0\},
}
 see \cite{DIK1} \red{for}\st{in} detail\red{s}. 
 A natural projection $\pi_+ : \UH \to \mathbb H^3$ is given by 
 $\pi_+ (\Vec{x}, \Vec{v}) = \Vec{x}$. 
\st{We now introduce notion of 
 the space}\red{Let} $\Geo$ \red{be the space} of all oriented 
 geodesics in $\mathbb{H}^3$ which is identified with the Grassmannian 
 manifold $\Gr$ of all oriented timelike planes 
 in $\mathbb{E}^{1,3}=\mathrm{SL}_{2}\mathbb{C}/\mathrm{GL}_1\mathbb{C}$.
 \st{It is known that} $\Geo$ can be represented as 
 \begin{equation}\label{eq:Gr}
 \Geo \cong \Gr = \mathrm{SL}_{2}\mathbb{C}/\mathrm{GL}_1\mathbb{C} 
 \cong \C P^1 \times \C P^1 \setminus \operatorname{diag}.
 \end{equation}
   Then we have a natural projection $\pi_0:\mathrm{U}\mathbb{H}^3\to 
 \Geo$ by
$\pi_0(\Vec{x},\Vec{v})=\Vec{x}\wedge \Vec{v}$.
 Finally there is also a natural projection from $\pi_* : \SL{\C} \to \UH$ by 
 $\pi_*(g) = (gg^*, g\Vec{e}_1g^*)$. Note that the projection 
 $\pi : \SL{\C} \to \mathbb H^3=\SL{\C}/\SU$
  is  given by the composition $\pi = \pi_+ \circ\pi_*$.
 \begin{figure}
 \end{figure}

 Using these projections, we will recall two Gauss maps for a surface in $\mathbb H^3$, which 
 are crucial in this paper. Note that several natural Gauss maps were already 
 defined in \red{see Appendix of the archive version} \cite{DIK1}.
\begin{definition}
 Let $f:M\to \mathbb{H}^3$ be a surface with unit normal $n$, and let 
 $\mathrm{Gr}_{1,1}(\mathbb{E}^{1,3}) = \Geo$.
 The {\rm Legendrian Gauss map} $G_{Le}$ and the {\rm Lagrangian Gauss map} $G_{La}$  of $f$ are 
 respective defined by 
 \eq{
G_{Le}=(f,n):M\to \mathrm{U}\mathbb{H}^3 \quad\mbox{and}\quad
G_{La}=f\wedge n:M\to \mathrm{Gr}_{1,1}(\mathbb{E}^{1,3}).
}
\end{definition}


 In \cite[Theorem A]{IK:CGCH3}, the following theorem has been proved.
\begin{theorem} Let $f: M \to \mathbb H^3$ be a surface with Gaussian curvature $K>-1$ and let $n$ be the 
 unit normal of $f$.
 Moreover let $G_{Le}=(f,n): M \to \UH$ and $G_{La} = f \wedge n: M \to \Geo$ be the Legendrian and Lagrangian 
 Gauss maps, respectively. Then the following statements are mutually equivalent$:$
 \begin{enumerate}
\item[$(1)$] The Gaussian curvature of $f$ is constant. 
\item[$(2)$] The Legendrian Gauss map is harmonic.
\item[$(3)$] The Lagrangian Gauss map is harmonic.
\item[$(4)$] The Klotz differential is holomorphic. 
\end{enumerate}
\end{theorem}

\subsection{The extended frames and CGC surfaces}\label{sbsc:extended}
 Let us consider the structure equations for a CGC $K>-1$ surface:
 \begin{equation}\label{eq:structures}
  \bar \partial \partial u + \frac{K}{2} (e^u -|Q|^2 e^{-u})=0
  \quad\mbox{and}\quad \bar \partial Q =0,     
 \end{equation}
 where we set
 \[
 \partial = \frac12\left(\frac{\partial}{\partial x} - i \frac{\partial}{\partial y}\right),\quad 
\bar  \partial = \frac12\left(\frac{\partial}{\partial x} + i \frac{\partial}{\partial y}\right).
 \]
 Note that, under a normalization $Q = 1$, the elliptic PDE \eqref{eq:structures}
 for $u$ becomes the famous \textit{sinh-Gordon} equation \cite{TKMil}. 
 
 If $u$ and $Q$ are solutions to \eqref{eq:structures}, then 
 $u$ and $\lambda^{-2} Q$ with any constant $\lambda \in S^1$
 are also solutions to \eqref{eq:structures}. Thus there exists 
 a family of CGC $K>-1$ surfaces $\{f^{\lambda}\}_{\lambda \in S^1}$
 in $\mathbb H^3$. Accordingly, there exists a family of unit normal vectors 
 $\{n^{\lambda}\}_{\lambda \in S^1}$ such that 
 \[
 G_{Le}^{\lambda} = (f^{\lambda},n^{\lambda}) : M \to \UH 
 \quad \mbox{and} \quad
  G_{La}^{\lambda} = f^{\lambda} \wedge n^{\lambda} : M \to \Geo 
 \]
 are a family of Legendrian and Lagrangian harmonic Gauss maps of $f^{\lambda}$, respectively.
 Then the family of CGC $-1<K=
 -1/\red{\cosh^2 (s/2)}<0$ surfaces in $\mathbb H^3$ has the following fundamental forms:
\begin{align}
 \mathrm{I}^{\lambda}&=\langle \d f^{\lambda}, \d f^{\lambda}\rangle=  \lambda^{-2}Q \> \dz^2+(e^u +|Q|^2e^{-u})\> \dz \dzb +\lambda^2 \bar Q\>\dzb^2, \label{eq:1a}\\
 \mathrm{I\!I}^{\lambda} &=\langle \d f^{\lambda}, \d  n^{\lambda} \rangle
 =\sigma (e^u -|Q|^2e^{-u})\>\dz \dzb,  \label{eq:2a}\\
 \mathrm{I\!I\!I}^{\lambda} &=\langle \d n^{\lambda}, \d n^{\lambda}\rangle=\sigma^2
  \left(-\lambda^{-2} Q \> \dz^2+(e^u +|Q|^2e^{-u})\> \dz \dzb
  -\lambda^2 \bar Q\>\dzb^2\right),\label{eq:3a}
\end{align}
 where 
 \begin{equation}\label{eq:sigma}
 \sigma = \tanh (s/2),\quad (s>0),
 \end{equation}
 and the Gaussian curvature $K$ of $f^{\lambda}$ is 
 \[
 -1<  K =-1 + \sigma^2 =  - \frac{1}{\cosh^2 (s/2)}<0.
 \]
  The $S^1$-parameter $\lambda$ has been called the \textit{spectral 
  parameter}, and the family of CGC $-1< K < 0$ surfaces 
  $\{f^{\lambda}\}_{\lambda \in S^1}$ in the above has been called the
  \textit{associated family} of $f$.
  Note that $\mathrm{I\!I}^{\lambda}$ is the same for any $\lambda \in S^1$, 
  that is $\mathrm{I\!I}^{\lambda} = \mathrm{I\!I}$.
  
 On the other hand, for a fixed positive number $s>0$, define two metrics parameterized by $\theta = \arg \lambda, (\lambda \in S^1)$ by 
 \[
 h_{\theta} := \frac1{\cosh^2(s/2)} \mathrm{I}^{\lambda} \quad \mbox{and} \quad
 h^{\star}_{\theta} := \frac1{\sinh^2(s/2)}  \mathrm{I\!I\!I}^{\lambda}. 
 \]
 Then $(\red{h_{\theta}}, \red{h_{\theta}^{\star}})$ is a pair of hyperbolic metrics. Moreover 
 by using the shape operator $B^{\lambda} = \left(\mathrm{I}^{\lambda}\right)^{-1} \mathrm{I\!I}^{\lambda}$,
 \[
 b_{\theta} = - \coth (s/2) \, B^{\lambda}
 \]
 is the Labourie operator for $(h_{\theta}, h_{\theta}^{\star})$.
 
 In \cite[Theorem B (1)]{IK:CGCH3}, it has been proved that 
 the family $\{G_{La}^{\lambda}\}_{\lambda \in S^1}$ of Lagrangian Gauss maps of $\{f^{\lambda}\}_{\lambda \in S^1}$ can be characterized as follows:
 \begin{theorem}[Theorem B (1) in \cite{IK:CGCH3}]
 Let $f^{\lambda}$ be the family of CGC $-1< K <0$ surfaces in $\mathbb H^3$ defined as
 above, and set $\blue{q \in \mathbb D^{\times}}$ such that 
 $\blue{|q| = \exp (-\operatorname{arcosh} \sqrt{-1/K})}$ holds.
 Then the map $G_{La}^{\lambda}|_{\lambda = q}$ is a harmonic local 
 diffeomorphism into $\mathbb H^2$.
 \end{theorem}

 Conversely, let $g: M \to \mathbb H^2$ be a harmonic map 
 and consider the extended frame 
 \[
 F^{\lambda} : \widetilde M \to \LISU,
 \]
 as explained in \red{the} introduction. 
 The extended frame is defined not only on $S^1$ but also on $\C^\times$.
  Then $F^{\lambda}|_{\lambda \in \C^{\times}}$ is a map into $\LSL$
  not $\LISU$.

  Let $\pi : \ISU \to \H^2$ be a natural projection. For the extended frame $F^{\lambda}$, 
 $\pi \circ F^{\lambda}|_{\lambda =1}$ is the original harmonic map $g$ up 
 to isometry and moreover, $g^{\lambda}= \pi \circ F^{\lambda}|_{\lambda \in S^1}$ gives 
 a family of harmonic maps in $\H^2$, the so-called {\rm associated family} of $g$.
 If $\hat Q\,\d z^2$ is the Hopf differential of $g$, then the Hopf differential of $g^{\lambda}$
 can be realized as $\lambda^{-2} \hat Q\, \d z^2$.
 \begin{theorem}[Theorem 2.1 in \cite{IK:CGCH3}]\label{thm:2.1}
  Let $F^{\lambda}$ be the extended frame of a harmonic map $g: M \to \mathbb H^2= \ISU/\Uone$
  and let $\blue{q = \exp (-t+ i \theta)} \in {\mathbb D}^{\times}$.
  Then 
    \[
  \blue{\hat f_{\blue{q}}}
  = F^{\lambda} (F^{\lambda})^*|_{\lambda = q}, \quad
  \blue{\hat n_{q}} =
F^{\lambda} \Vec{e}_1 (F^{\lambda})^*|_{\lambda = q}, 
  \]
   is a family of CGC surfaces in $\mathbb H^3$ with the constant Gaussian curvature 
\begin{equation}\label{eq:K}
K = - \left(\frac{2 |q|}{|q|^2 + 1}\right)^2 = -\frac1{\cosh^2 t} \in (-1, 0)
\end{equation}
 and a family of unit normals $\hat n_{q}$. 
 \blue{The Klotz differential of $\hat f_q$ is given by 
\begin{equation}\label{eq:familyKlotz}
 Q^{q}\, \d z^2  = e^{-2i \theta}\left(\frac{|q|^2 + 1}{2 |q|}\right)^2 \hat Q\, \d z^2. 
\end{equation}
 }
\st{ Moreover, the Klotz differential 
 of $f^{\lambda}$
 is given by }
\end{theorem}
\blue{
 By abuse of notation, the CGC $K = -1/\cosh^2 t$ surface $\hat f_{q}$ 
 will be called the \textit{spectral parameter deformation},
 and the $S^1$-family $\{ \hat f_{q}\}_{q = |q| e^{i\theta}}$ with fixed $|q|$
 becomes the associated family. Therefore there exists the associated family 
 $\{f^{\lambda}\}_{\lambda \in S^1}$ of some CGC $K = -1/\cosh^2 (s/2)$ with $s = 2 t$ and 
 $Q = Q^q$ surface $f$
 such that 
 \[
 f^{\lambda} = \hat f_{q}
 \]
 holds up to rigid motion.
 }
 Moreover, the Lagrangian Gauss map can be represented by $F^{\lambda}$ as
\begin{equation}\label{eq:familyLa}
 G_{La}^{\lambda} =  \blue{ \hat f_{q}\wedge \hat n_{q}}=
 F^{\lambda} \Vec{e}_1 (F^{\lambda})^{-1}|_{\lambda= q \in {\mathbb D}^{\times}}.
\end{equation}
 \begin{remark}
 From \red{\eqref{eq:K} and} \eqref{eq:familyKlotz}, the family of CGC surfaces $\hat f_{q}$
 in {\rm Theorem \ref{thm:2.1}} has a symmetry, that is,  $\hat f_{q}$ and $\hat f_{1/q}$
 are the same CGC surface up to rigid motion.
\end{remark}
\subsection{The complex landslides flow and CGC surfaces}\label{sbsc:cland}
 Recall that the landslide \eqref{eq:landslide} is given by 
\begin{equation*}
L_{\theta} (h, h^{\star}) :=  (h (\beta_{\theta} , \beta_{\theta}),
h (\beta_{\theta+\pi} , \beta_{\theta+ \pi})),
\end{equation*}
where $\beta_{\theta} := \cos (\theta/2) E + \sin (\theta/2) J b$, and 
$b$ is the Labourie operator for $(h, h^{\star})$, see \red{the} introduction.
Note that $J$ is the complex structure compatible with \st{respect to} $h$
and it is not related to the complex structure given in Section \ref{sbsc:Gauss}.
We recall the following fundamental lemma, see \cite{BMS:cyc1}.
\begin{lemma}\label{lem:CGC}
 Let $(h, h^{\star}) \in \mathcal T \times \mathcal T$ and a positive constant $s>0$. 
 Then there exists 
 a unique CGC  $-1< K = -1/\cosh^2 (s/2) <0$ surface $f$ in $\mathbb H^3$ such that 
 the first and third fundamental forms are respectively given as
\begin{equation}\label{eq:fundhh}
 \mathrm{I} = \cosh^2 (s/2) h \quad\mbox{and}\quad \mathrm{I\!I\!I} = \sinh^2 (s/2) h^{\star}.
\end{equation}
 Conversely, for a given CGC $-1< K= -1/\cosh^2(s/2) < 0$ surface $f$ in $\mathbb H^3$, define 
 $h$ and $h^{\star}$ by \eqref{eq:fundhh}. Then $h$ and $h^{\star}$ are hyperbolic metrics.
 \end{lemma}
 \begin{proof}
 For a 
 given $(h, h^{\star})$, there exists a unique Labourie operator $b$ 
 such that $h(b\bullet, b\bullet) = h^{\star}$, see \red{the} introduction.
 Define 
 \[
 \mathrm{I} = \cosh^2 (s/2)h, \quad B = - \tanh (s/2) b.
 \]
 Note that $\det B = \tanh^2 (s/2)$ holds.

 Then from the properties \red{(2), (3)} of the Labourie operator $b$
 \red{in the introduction (A)}, 
 the pair $(\mathrm{I}, B)$ satisfies 
 \begin{equation}\label{eq:GC}
 \d^\nabla B=0, \quad K = -1 + \det B,     
 \end{equation}
 where $\nabla$ is the Levi-Civita connection of $\mathrm{I}$ 
 and $K = -1/\cosh^2(s/2)$ is the curvature of $\mathrm{I}$.  
 Since $ (\d^\nabla B) (u, v) = 0$ is \st{equivalent to} 
 \[
 \nabla_{u} (B v) -\nabla_{v} (B u) = B[u, v],
 \]
 which is the Codazzi equation for a surface in $\mathbb H^3$, see \cite[p.138]{doCarmo}. \red{Moreover,
 $B$ is self-adjoint with respect to $h$ by 
 (1) in the introduction (A)}.
 Since the second equation in \eqref{eq:GC} is 
 \red{just} the Gauss equation, thus there
 exists a unique surface $f$ \red{up to rigid motion}
 such that the first fundamental form is $\mathrm{I}$ and
 the second fundamental form $\mathrm{I\!I}$ is  $\mathrm{I\!I} (u, v):= \mathrm{I} (u, B v)$. 
 \st{Moreover}\red{Then}, the third fundamental form  $\mathrm{I\!I\!I}$ can be computed as
 \begin{align*}
 \mathrm{I\!I\!I}(u, v) 
 &= \mathrm{I} (B u, B v) \\
 &= \tanh^2(s/2) \mathrm{I} (b u, b v)\\
 &= \sinh^2(s/2) h(b u, b v)\\ 
 &= \sinh^2(s/2) h^{\star}(u, v), 
 \end{align*}
 and the claim follows.
 The converse statement is clear and this completes the proof.
  \end{proof}
 The complex landslide flow $P:{\overline \D}^{\times} \times \mathcal T \times \mathcal T 
 \to \mathcal P$ is defined by $P (q, h, h^{\star}) =
  SGr_{s}(L_{\theta}(h, h^{\star}))$ with $s = - \log |q|$ and $\theta = -\arg q$,
  where $SGr_{s}$ is the smooth grafting, which has been explained in \red{the} introduction:
\begin{enumerate}
    \item \red{Begin by considering}\st{Let} the pair of hyperbolic metrics by the landslide $L_{\theta} (h, h^{\star})$.
    \item By Lemma \ref{lem:CGC}, there exists a CGC surface $f$ with \red{Gaussian curvature in}
    $-1< K=-1/\cosh^2 (s/2) <0$. 
    \item \red{Consider the geodesic whose initial velocity is the unit normal $n(p)$ of $f(p)$, and take the point where it
    intersects the boundary of $\mathbb H^3 \cong \mathbb C P^1$.}
  \item It gives \red{a}\st{the} developing map 
 \begin{equation}\label{eq:developing}
  \operatorname{dev}_s: \widetilde M \to \mathbb C P^1,
 \end{equation}
  \red{which}\st{and it} defines a complex projective structure on $M$.  
\end{enumerate}
 \red{It is proved that}
 the complex landslide flow $P$ is holomorphic with respect to 
 $\overline {\mathbb  D}^{\times}$ by showing that the holonomy of 
 the developing map is holomorphic,
 \cite[Theorem 5.1]{BMS:cyc1}.

\section{Proof of the main theorem}\label{sc:proof}
 We now prove Theorem \ref{thm:main} in \red{the} introduction.
 The crucial observation is the following lemma about a relation of the complex structure $J$ \red{of}\st{on} 
 the first fundamental form $\mathrm{I}$ and the complex structure, which is denoted by $z = x + \red{y i}$,
 \red{in}\st{on} the conformal class of the second fundamental form $\mathrm{I\!I}$.
\begin{lemma}\label{lem:J}
 Let $f$ be a CGC $-1< K = -1/ \cosh^2 (s/2) <0$ surface in $\mathbb H^3$ with 
 the fundamental forms $\mathrm{I}, \mathrm{I\!I}$ and 
  $\mathrm{I\!I\!I}$ in \eqref{eq:1}, \eqref{eq:2} and \eqref{eq:3}, respectively. 
 Define an endomorphism $J: TM \to TM$ by 
\begin{equation}\label{eq:J}
\left\{
\begin{array}{l}
\displaystyle
 J \partial_z := i \coth (s/2) B \partial_{z} \\[0.1cm]
 \displaystyle
 J \partial_{\bar z} : =  -  i \coth (s/2) B \partial_{\bar z}
\end{array},
\right.
\end{equation}
 where $B$ is the shape operator of $f$.
 Then $J$ is the unique complex structure on $M$ compatible with the first fundamental form 
 $\mathrm{I}$.
\end{lemma}
\begin{proof}
 \red{Since $\mathrm{I\!I}(\bullet, \bullet) = \mathrm{I} (\bullet, B \bullet)$, $\mathrm{I\!I\!I}(\bullet, \bullet) = 
 \mathrm{I} (B \bullet, B \bullet)$, we have}
  \st{By a straightforward computation,} 
 \begin{align*}
 \mathrm{I} (J \partial_z, J \partial_z) &=    - \coth^2 (s/2)\mathrm{I} ( B \partial_z, B \partial_z)  = - \coth^2 (s/2)\mathrm{I\!I\!I} ( \partial_z, \partial_z) =Q 
 = \mathrm{I} (\partial_z, \partial_z), \\
 \mathrm{I} (J \partial_{\bar z}, J \partial_{\bar z}) &=    - \coth^2(s/2)\mathrm{I} ( B \partial_{\bar z}, B \partial_{\bar z})  
  = - \coth^2 (s/2)\mathrm{I\!I\!I} ( \partial_{\bar z}, \partial_{\bar z}) =\bar Q 
 = \mathrm{I} (\partial_{\bar z}, \partial_{\bar z}), \\
 \intertext{and}
 \mathrm{I} (J \partial_{z}, J \partial_{\bar z}) &= 
 \coth^2 (s/2) \mathrm{I} ( B \partial_{z}, B \partial_{\bar z})  
  = \coth^2 (s/2) \mathrm{I\!I\!I} ( \partial_{z}, \partial_{\bar z}) =\frac12 (e^u + |Q|^2e^{-u}) = \mathrm{I} (\partial_{z}, \partial_{\bar z})
 \end{align*}
 hold, \red{where $Q\,\d z^2$ is the Klotz differential and $H$ is the mean curvature in \eqref{eq:H}.} Thus $J$ is compatible with $\mathrm{I}$. Moreover, \red{by using the above equations, 
 we have explicit formulas of $B \partial_z$ and
 $B \partial_{\bar z}$, and}\st{a straightforward computation shows that} 
\begin{equation}\label{eq:Jexplicit}
\begin{array}{l}
\displaystyle
 J \partial_z = \frac{2 iH}{\tanh(s/2)}  \partial_z -  \frac{2 i  Q}{e^u - |Q|^2e^{-u}} \partial_{\bar z},\\[0.1cm]
 \displaystyle
 J \partial_{\bar z}  =  \frac{2 i  \bar Q}{e^u - |Q|^2e^{-u}} \partial_{z}
- \frac{2 iH}{\tanh (s/2)}  \partial_{\bar z}
\end{array}
\end{equation}
 \red{follow.}
 From \red{the form} \st{the definition of} $J$ in \red{\eqref{eq:Jexplicit}}, it is easy to see that 
 \[
 J^2 \partial_z = - \partial_z\quad \mbox{and} \quad 
J^2 \partial_{\bar z} = - \partial_{\bar z}.
 \]
 It is also easy to see that $J$ preserves the orientation, and thus $J$ is the unique complex structure on $M$ 
 compatible with $\mathrm{I}$.
This completes the proof.
\end{proof}
\begin{corollary}\label{coro:betatheta}
 Let $\beta_{\theta}$ be the family of operators  by 
 $\beta_{\theta}= \cos(\theta/2) E + \sin (\theta/2) J b, \, (\theta \in [0, 2 \pi))$
 of the landslide flow in \eqref{eq:landslide}.
 Moreover, set $\lambda^{1/2} = \cos(\theta/2) + i \sin (\theta/2)\in S^1$.
 Then the following relations hold$:$
\begin{equation}\label{eq:beta}
\beta_{\theta} \partial_z =  \lambda^{-1/2} \partial_z\quad \mbox{and}\quad 
\beta_{\theta} \partial_{\bar z} =  \lambda^{1/2} \partial_{\bar z}.
\end{equation}
\end{corollary}
\begin{proof}
  A straightforward computation by using \eqref{eq:J} shows that 
\[
\beta_{\theta} \partial_z = \cos (\theta/2) \partial_z + \sin (\theta/2) J b \partial_{z}
= \cos (\theta/2) \partial_z -i \sin (\theta/2) \partial_{z} = \lambda^{-1/2} \partial_{z}
\]
holds. Similarly $\beta_{\theta} \partial_{\bar z} = \lambda^{1/2} \partial_{\bar z}$ holds. 
\end{proof}

Combining the landslide flow and \red{Corollary \ref{coro:betatheta}}\st{Lemma}, we have the following proposition.
\begin{proposition}\label{prp:associated}
 \blue{For a given landslide flow $(h_{\theta}, h^{\star}_{\theta})=L_{\theta}(h, h^{\star}) 
 \in \mathcal T \times \mathcal T, \, (\theta \in [0, 2 \pi))$ and a positive constant $s>0$,
 let $\tilde f_{q}\, (q = \exp(-s + i \theta))$ be the corresponding CGC 
 surface with $K = -1/\cosh^2 (s/2)$ in $\mathbb H^3$ given by {\rm Lemma \ref{lem:CGC}}.
  Then there exists a CGC  surface $f$ with $K = -1/\cosh^2 (s/2)$ and 
  its associated family $\{f^{\lambda}\}_{\lambda \in S^1}$ 
  $($where the conformal structure is given by the second fundamental form of $f)$
  in $\mathbb H^3$ such that 
   \[
  \tilde f_q =f^{\sqrt{\lambda}}\quad (\lambda = \exp (i \theta) \in S^1)
 \]
 holds up to rigid motion.}
\end{proposition}
\begin{proof}
First recall that $(h_{\theta}, h_{\theta}^{\star}) = L_{\theta} (h, h^{\star})$ 
 is given by 
\[
(h_{\theta}, h_{\theta}^{\star}) = (h(\beta_{\theta}, \beta_{\theta}), 
  h^{\star}(\beta_{\theta+\pi}, \beta_{\theta+\pi})), 
\quad 
\mbox{where $\beta_{\theta} = \cos (\theta/2) E + \sin (\theta/2) J b$},
 \]
 and $b$ is the Labourie operator between $h$ and $h^{\star}$.
The first and third fundamental forms of $\tilde f_{q}$ are given by 
\[
(\mathrm{I}_{q}, \mathrm{I\!I\!I}_{q})
  = (\cosh^2 (s/2)h_{\theta}, \sinh^2 (s/2)h^{\star}_{\theta}).
\]
 Moreover, the second fundamental form is $\mathrm{I\!I}_{\theta}
 =\mathrm{I}_{q}(\bullet, -\tanh(s/2)b_{\theta}\bullet)$, where 
 $b_{\theta}$ is the Labourie operator between $h_{\theta}$ and $h_{\theta}^{\star}$, 
 that is, $h_{\theta}^{\star} = h_{\theta} (b_{\theta} \bullet, b_{\theta} \bullet)$
 holds, see Lemma \ref{lem:CGC}.
 For $q_0  = \exp (-s)\in (0,1)$, we introduce the coordinates $z = x + \red{y i}$
 such that the first, second and third fundamental forms $\mathrm{I}_{q_0},\mathrm{I\!I}_{q_0}$ and $\mathrm{I\!I\!I}_{q_0}$ of $\tilde f_{q_0}$ can be represented by \eqref{eq:1}, \eqref{eq:2} and \eqref{eq:3}, that is,
 \[
(\mathrm{I}_{q_0}, \mathrm{I\!I}_{q_0}, \mathrm{I\!I\!I}_{q_0})= \cosh^2 (s/2) (h, \,  h(\bullet, B \bullet), \, h(B \bullet, B \bullet)) 
 \]
 holds, where we set $B = - \tanh (s/2) b$ with $b= b_{0}$. 
 It has been \red{shown}\st{known} \cite[Lemma 3.5]{BMS:cyc1} that the Labourie operator $b_{\theta}$ is explicitly given by $b_{\theta}= \beta_{-\theta} \circ b \circ \beta_{\theta}$. 
 
 To show that a CGC surface $\tilde f_{q}\, (q= \exp(-s + i \theta))$ constructed by Lemma \ref{lem:CGC} is the associated family $f^{\sqrt{\lambda}}\, (\lambda \in S^1)$ in Section \ref{sbsc:extended},
 we compare $\mathrm{I}_{q}, \mathrm{I\!I}_{q}$ and $\mathrm{I\!I\!I}_{q}$ and $\mathrm{I}^{\lambda}$ in \eqref{eq:1a}, $\mathrm{I\!I}^{\lambda}$ in \eqref{eq:2a} and $\mathrm{I\!I\!I}^{\lambda}$ in \eqref{eq:3a}, respectively; 
 Set $\lambda = e^{i \theta} \in S^1$.
 By Corollary \ref{coro:betatheta},
\begin{align*}
\mathrm{I}_{\theta} (\partial_z, \partial_z) &:= 
\cosh^2 (s/2) h(\beta_{\theta}\partial_z, \beta_{\theta}\partial_z) \\ 
&= \cosh^2 (s/2) h(\lambda^{\red{-}1/2}\partial_z, \lambda^{\red{-}1/2}\partial_z) \\
&= \lambda^{\red{-1}} \cosh^2 (s/2) h(\partial_z, \partial_z) \\
&= \mathrm{I}^{\sqrt{\lambda}}(\partial_z, \partial_z)
\end{align*}
holds.
Similarly, $\mathrm{I}_{\theta} (\partial_{\bar z}, \partial_{\bar z})  = \mathrm{I}^{\sqrt{\lambda}} (\partial_{\bar z}, \partial_{\bar z})$ 
follows. Moreover, 
\begin{align*}
\mathrm{I}_{\theta} (\partial_z, \partial_{\bar z}) 
&:= \cosh^2 (s/2) h(\beta_{\theta}\partial_z, \beta_{\theta}\partial_{\bar z}) \\ 
&= \cosh^2 (s/2) h(\lambda^{\red{-}1/2}\partial_z, \lambda^{1/2}\partial_{\bar z}) \\
&= \cosh^2 (s/2) h(\partial_z, \partial_{\bar z}) \\
&= \mathrm{I}^{\sqrt{\lambda}}(\partial_z, \partial_{\bar z})
\end{align*}
holds, therefore $\mathrm{I}_{\theta} = \mathrm{I}^{\sqrt{\lambda}}$.
The same argument can be applied to $\mathrm{I\!I\!I}_{\theta}$, and it follows  $\mathrm{I\!I\!I}_{\theta}= \mathrm{I\!I\!I}^{\sqrt{\lambda}}$.
Next the second fundamental form $\mathrm{I\!I}_{\theta}$  can be computed as
\[
\mathrm{I\!I}_{\theta}(\partial_z, \partial_{z})
= \mathrm{I}_{\theta}(\partial_z, B_{\theta} \partial_{z}).
\]
 Since $b_{\theta} = \beta_{-\theta} \circ b \circ \beta_{\theta}$, \st{thus} the 
 shape operator $B_{\theta}$ is given by $\beta_{-\theta} \circ B \circ \beta_{\theta}$, and thus
\begin{align*}
 \mathrm{I\!I}_{\theta}(\partial_z, \partial_{z}) &= 
 \mathrm{I}_{\theta}(\partial_z,  \beta_{-\theta} \circ B \circ \beta_{\theta} \partial_{z})  \\
 &=  \mathrm{I}(\beta_{\theta} \partial_z, B \circ \beta_{\theta} \partial_{z})  \\
 &= \lambda^{\red{-1}} \mathrm{I}(\partial_z,   B\partial_{z})  \\
 &=0= \mathrm{I\!I}^{\sqrt{\lambda}}(\partial_z, \partial_{z})
\end{align*}
holds. Similarly $\mathrm{I\!I}_{\theta}(\partial_{\bar z}, \partial_{\bar z}) =0= \mathrm{I\!I}^{\sqrt{\lambda}}(\partial_{\bar z}, \partial_{\bar z})$
holds. Finally
\begin{align*}
 \mathrm{I\!I}_{\theta}(\partial_z, \partial_{\bar z}) &= 
 \mathrm{I}_{\theta}(\partial_z,  \beta_{-\theta} \circ B \circ \beta_{\theta} \partial_{\bar z})  \\
 &=  \mathrm{I}(\beta_{\theta} \partial_z, B \circ \beta_{\theta} \partial_{\bar z})  \\
 &= \mathrm{I}(\partial_z,   B\partial_{\bar z})  \\
 &= \mathrm{I\!I}^{\sqrt{\lambda}}(\partial_z, \partial_{\bar z}).
\end{align*}
Thus $ \mathrm{I\!I}_{\theta} =  \mathrm{I\!I}^{\sqrt{\lambda}}$ also follows, and 
therefore the two families of CGC surfaces $\{\tilde f_{q}\}_{q =|q| e^{i\theta}}  \;
(\theta \in [0, 2\pi))$ 
and $\{f^{\sqrt{\lambda}}\}_{\lambda \in S^1}$ are the 
same up to isometry. \st{Since $f= \tilde f_{q_0} =f^{\sqrt{\lambda}}|_{\lambda =1}$ 
thus the two families exactly coincide.} This completes the proof.
\end{proof}
\begin{remark}\blue{In the proof of {\rm Proposition \ref{prp:associated}}, 
 we showed the equivalence of the first, second and 
 third fundamental forms, respectively. In fact the equivalence of 
 two of them are sufficient, since surfaces are convex.}
\end{remark}
 We now look at a relation between the complex landslide and \blue{the spectral
 parameter deformation in 
 Theorem \ref{thm:2.1}.}\st{Lagrangian Gauss map.}
\begin{corollary}\label{coro:complexland}
 \blue{Retain the assumptions in {\rm Proposition \ref{prp:associated}}. Moreover let  $\hat f_{\blue{q}}$ be the spectral parameter deformation in {\rm Theorem \ref{thm:2.1}} corresponding to $f^{\lambda}$.
 Then 
 \[
  \tilde f_q =\hat f_{\sqrt{q}}, \quad (q = \exp (-s + i \theta) \in \mathbb D^{\times}),
 \]
 holds. Moreover,  the complex landslide $P_{q}(= SGr \circ L), \, (q =\exp(-s + i \theta))$ is given by the associated family
 of the Lagrangian Gauss map of $\hat f_{\sqrt{q}}$  and vice versa.
 }
\end{corollary}
\begin{proof}
 The first statement follows from Theorem \ref{thm:2.1}.
 Let us consider the smooth grafting $SGr_{s}$ of the landslide $L_{\theta}(h, h^{\star})$.
 By Proposition \ref{prp:associated}, there exists the family of corresponding CGC $-1< K = -1/\cosh^2(s/2) <0$ surfaces, which is the 
 associated family $\{f^{\sqrt{\lambda}}\}_{\lambda \in S^1}$ with $\lambda = e^{i\theta}$.
  Then the developing map $\operatorname{dev}_s: \widetilde M \to  \C P^1$ 
  of $SGr_{s}$ is given as in \eqref{eq:developing}.
 
 \st{On the one hand,} The Lagrangian Gauss map $G_{La}^{\lambda} 
 = f^{\lambda} \wedge n^{\lambda}$ takes values in 
  \begin{equation*}
 \Geo \cong \Gr = \mathrm{SL}_{2}\mathbb{C}/\mathrm{GL}_1\mathbb{C} 
 \cong \C P^1 \times \C P^1 \setminus \operatorname{diag},
 \end{equation*}
 \red{which is a pair of intersecting points of the geodesic from the point 
 $f^{\sqrt{\lambda}}(p)$ with initial velocity $\pm n(p)$ to the pair of boundaries of
 $\mathbb H^3 \cong \mathbb C P^1$,}
 and thus $\operatorname{dev}_s$ is given by 
 $\operatorname{pr}_1 \circ G_{La}^{\lambda}$, where 
 $\operatorname{pr}_1$ denotes the projection to the first component.
 Moreover, $G_{La}^{\lambda}$ can be represented by 
 \begin{equation}\label{eq:LaGauss}
G_{La}^{\lambda} = \hat f_{\sqrt{q}} \wedge \hat n_{\sqrt{q}} 
= (F^{\lambda})^{-1} \Vec{e}_1F^{\lambda}|_{\lambda = \sqrt{q}}, \quad (q= 
\exp (-s + i \theta) \in \mathbb D^{\times}).
 \end{equation}
 Therefore the Lagrangian Gauss map gives 
 the developing map $\operatorname{dev}_s$ and thus the complex landslide.
 The converse statement is also clear.
 This completes the proof.
 \end{proof}
 By Corollary \ref{coro:complexland}, the holonomy of the 
 complex landslide $P_q = SGr_{s}\circ L_{\theta}$ with $s= - \log|q|, 
 \theta = \arg q$ 
 is given by the holonomy of the \red{extended} frame $F^{\lambda}$ evaluated at $\lambda = \sqrt{q}$, and the statement of Theorem \ref{thm:main} follows.
 We are now ready to prove  Corollary \ref{coro:landslide} in \red{the} introduction.
\begin{proof}[Proof of {\rm Corollary \ref{coro:landslide}}]
 To see the holomorphic dependence with respect to $\lambda$,
 we consider the gauged extended frame:
 \[
F^{\lambda} \Vec{e}_1(F^{\lambda})^{-1}|_{\lambda = \sqrt{q}}
= F^{\lambda} D^{\lambda} \Vec{e}_1(F^{\lambda} D^{\lambda})^{-1}|_{\lambda = \sqrt{q}}
= D^{\sqrt{\lambda}}
\hat F^{\lambda} \Vec{e}_1(\hat F^{\lambda})^{-1}
 D^{1/\sqrt{\lambda}}\big|_{\lambda = q},
 \]
 where $D^{\lambda} = \diag (\lambda^{-1/2}, \lambda^{1/2})$ and we set $\hat F^{\lambda} = (D^{1/\lambda} F^{\lambda} D^{\lambda})|_{\lambda =\sqrt{\lambda}}$. 
 Recalling that $\alpha^{\lambda} = (F^{\lambda})^{-1} \d F^{\lambda} = 
 \lambda^{-1} \alpha_{\mathfrak p}^{\prime} + \alpha_{\mathfrak k}
+ \lambda \alpha_{\mathfrak p}^{\prime \prime}$,
 it is straightforward to see that 
\begin{equation}\label{eq:alphahat}
\hat \alpha^{\lambda} := (\hat F^{\lambda})^{-1}\d \hat F^{\lambda}
= \lambda^{-1} {\alpha_{\mathfrak p}^{\prime}}_{21} + {\alpha_{\mathfrak p}^{\prime}}_{12}
 + \alpha_{\mathfrak k}
+ \lambda {\alpha_{\mathfrak p}^{\prime \prime}}_{12}
+ {\alpha_{\mathfrak p}^{\prime \prime}}_{21}, 
\end{equation}
 where $x_{ij} (i, j \in \{1, 2\})$ denotes the $(i j)$-entry for a $2$ by $2$ matrix valued 
 $1$-form $x$. Therefore $\hat F^{\lambda}$ 
 is holomorphic with respect to $\lambda \in \mathbb D^{\times}$, and moreover
\[
\widehat {\operatorname{dev}}_s = 
\operatorname{pr}_1 \circ \left(\hat F^{\lambda} \Vec{e}_1(\hat F^{\lambda})^{-1}\right)\Big|_{\lambda = q}
\]
 defines the same complex projective structure given by $\operatorname{dev}_s$.
 The gauged frame $\hat F^{\lambda}$ will be called the \textit{untwisted extended frame}, and 
 we set $\widehat \nabla^{\lambda} = \d  + \hat \alpha^{\lambda}$.
 
 Then the holonomy of the untwisted extended frame of $\widehat \nabla^{\lambda}$
 is given by 
 \[
  \hat H^{\lambda}: \pi_1 (M) \to \LSLU
 \]
 with $\gamma^* \hat F^{\lambda} = \hat H^{\lambda} \hat F^{\lambda}$ for $\gamma \in \pi_1 (M)$ and $\lambda \in \mathbb D^{\times}$.
 Clearly the holonomy $\hat H^{\lambda}$ is holomorphic with respect to the spectral parameter $\lambda \in  
 \mathbb D^{\times}$, and thus the holonomy of the developing map $\widehat{\operatorname{dev}}_s$ is 
 holomorphic with respect to $\lambda$ in $\mathbb D^{\times}$.
 The limit $s\to 0$ to the holonomy of the developing map $\widehat{\operatorname{dev}}_s$ gives the 
 holonomy at $\lambda \in S^1$.
 This completes the proof.
\end{proof}
\begin{remark}
 It is well known \cite{BRS:Min} that the relation 
 \[
 F^{\lambda} \longleftrightarrow 
 \hat F^{\lambda} = (D^{1/\lambda} F^{\lambda} D^{\lambda})|_{\lambda =\sqrt{\lambda}}
  \quad  \mbox{with}\quad  D^{\lambda}= \diag (\lambda^{1/2}, \lambda^{-1/2})
 \]
 is an isomorphism  between twisted and untwisted loop groups.  \red{Note that,}
 $F^{\lambda}$ belongs to the twisted loop group $\LSL$ with respect to $\tau$, 
 that is, $F^{-\lambda} = \tau F^{\lambda}$ holds, while 
 $\hat F^{\lambda}$ belongs to \red{the} untwisted loop group $\LSLU$ and 
 it does not satisfy such \red{a} relation.
\end{remark}
\bibliographystyle{plain}

 \end{document}